\documentclass[12pt]{article}

\usepackage{amssymb,amsmath,amsfonts,eurosym,geometry,ulem,graphicx,caption,color,setspace,sectsty,comment,footmisc,caption,pdflscape,subfigure,array,hyperref,subcaption,bbm}
\usepackage{tikz}
\usepackage[numbers]{natbib}

\newtheorem{theorem}{Theorem}
\newtheorem{assumption}{Assumption}
\newtheorem{definition}{Definition}

\newtheorem{lemma}[theorem]{Lemma}
\newtheorem{proposition}{Proposition}
\newenvironment{proof}[1][Proof]{\noindent\textbf{#1.} }{\ \rule{0.5em}{0.5em}}

\numberwithin{equation}{section}

\newcolumntype{L}[1]{>{\raggedright\let\newline\\arraybackslash\hspace{0pt}}m{#1}}
\newcolumntype{C}[1]{>{\centering\let\newline\\arraybackslash\hspace{0pt}}m{#1}}
\newcolumntype{R}[1]{>{\raggedleft\let\newline\\arraybackslash\hspace{0pt}}m{#1}}

\begin{document}

\begin{titlepage}
\title{Nonexplosion for a large class of superlinear stochastic parabolic equations, in arbitrary spatial dimension}

\author{
Michael Salins\thanks{Email: \href{mailto:msalins@bu.edu}{msalins@bu.edu}} \\ 
    Department of Mathematics and Statistics, Boston University
    \and
    Yuyang Zhang\thanks{Email: \href{mailto:yyz@bu.edu}{yyz@bu.edu}} \\ 
    Questrom School of Business, Boston University 
}

\date{\today}
\maketitle
\begin{abstract}
This paper explores the finite time explosion of the stochastic parabolic equation $\frac{\partial u}{\partial t}(t,x)=Au(t,x)+\sigma(u(t,x))\dot{W}(t,x)$ in arbitrary bounded spatial domain with a large class of space-time colored noise under Neumann, periodic or Dirichlet boundary conditions where $A$ is second-order self-adjoint elliptic operator and $\sigma$ grows like $\sigma(u)\approx C(1+|u|^{\chi})$  where $\chi=1+\frac{1-\eta}{2\beta}$ with $\eta$ and $\beta$ are the parameters related to the singularities of heat kernel and noise covariance kernel. We improve upon previous results by proving the theory in arbitrary spatial dimension, general elliptic operator, general space-time colored noise, a larger class of boundary conditions and proves that $\chi$ can reach the level $1+\frac{1-\eta}{2\beta}$.
\end{abstract}
\setcounter{page}{0}
\thispagestyle{empty}
\end{titlepage}
\pagebreak \newpage
\section{Introduction}
We explore whether solutions to a stochastic heat equation with superlinear multiplicative noise explode in finite time. The equation is
\begin{align}
    \begin{cases}
        \frac{\partial u}{\partial t}(t,x)=Au(t,x)+\sigma(u(t,x))\dot{W}(t,x), x\in D,t>0\\
        Bu(t,x)=0, x\in \partial D, t>0\\
        u(0,x)=u_0(x)\in L^\infty(D), u_0\geq 0.
    \end{cases},\label{eq1}
\end{align}
where $A$ is a self adjoint elliptic operator satisfying the uniform ellipticity condition. $D\subset\mathbb{R}^d$ is a bounded open domain with boundary condition given by $Bu$, which can be Neumann, periodic or Dirichlet. We believe our method applies to a larger class of boundary conditions; however, to keep the argument clear, we restrict our focus to these three, as they are particularly relevant in many contexts. Additionally, in this paper, we focus on positive solutions by assuming $\sigma(0)=0$. Combined with the condition $u_0\geq 0$, this ensures the solution remains positive due to the comparison principle \cite{Kotelenez1992,Mueller199101}.

In a series of works \cite{Mueller1993,Mueller1991,Muller1998,Mueller2000}, Mueller and Sowers established a critical growth rate on $\sigma(u)$ in the case where $A$ is Laplacian, $Bu$ is periodic, $\dot{W}$ is the space-time white noise and spatial dimension $d=1$. If $|\sigma(u)|\leq C(1+|u|^{\gamma})$ for $\gamma<\frac{3}{2}$ then solutions cannot explode in finite time with probability 1. Conversely, if $\sigma(u) =c|u|^{\gamma}$ for some $c>0$ and $\gamma>\frac{3}{2}$,
solutions will blow up with positive probability. In the critical case $\gamma=\frac{3}{2}$, solutions cannot explode\cite{salins2025}.

In higher dimensions, the introduction of colored noise is necessary to ensure that solutions remain well-defined as functions, as demonstrated by \cite{DaPrato_Zabczyk_2014,Dalang1999}. Franzova \cite{Franzova1999} extended Mueller's result to higher dimensions with spatially homogeneous Riesz kernel noise, considered the case $\sigma(u)=u^{\gamma}$ under Dirichlet boundary conditions, and showed the existence result when $\gamma<1+\frac{1-\alpha/2}{d}$. Here, $\alpha$ represents the size of correlation of the colored noise.

In this paper, we extend Franzova’s result by by allowing $\gamma$ to reach $1+\frac{1-\alpha/2}{d}$ and prove an analogous non-explosion result in any spatial dimension, a larger class of driving noises and we allow Neumann, periodic or Dirichlet boundary conditions. While we conjecture that $1+\frac{1-\alpha/2}{d}$ is also a critical threshold, we leave the proof of positive probability explosion for future work.

Building on \cite{Franzova1999,salins2025}, we extend the results by considering bounded domains $D\subset\mathbb{R}^d$ with Neumann, periodic or Dirichlet boundary conditions in arbitrary dimension($d\geq 1$).

Similar results have been extended to various settings, including the reaction–diffusion
equation on unbounded domains\cite{Krylov1996}, fractional heat equations\cite{Foondun2019,Bezdek2018},nonlinear Schrödinger
equation\cite{Debuss2002} and stochastic wave equation\cite{Mueller1997}. More recently, researchers have explored
how adding superlinear deterministic forcing terms $f(u(t,x))$ to the right side of \eqref{eq1} affects the finite-time explosion behavior of the stochastic heat equation  \cite{AGRESTI2023,chen.huang:23:superlinear,FERNANDEZBONDER2009,Foondun2015,Liang2022,Foondun2021,Dalang2019,salins2024,SHANG2022}. Similar explosion phenomena have been studied in the context of the stochastic wave
equation\cite{FOONDUN2022,MILLET2021175}. In our setting, we focus on the case where $f\equiv 0$.

 Recall the standard approach to construct a local unique mild solution. For any $n>0$ define 
\begin{align}
    \sigma_n(u)=\begin{cases}
        \sigma(-n), \textrm{ if }u<-n\\
        \sigma(u), \textrm{ if }u\in [-n,n]\\
        \sigma(n), \textrm{ if }u>n.
    \end{cases}.\label{def:u_n}
\end{align}
Then for each $n$, $\sigma_n$ is globally Lipschitz continuous and there exists a unique global mild solution
\begin{align}
    u_n(t,x)=\int_D G(t,x,y)u(0,y)dy+\int_0^t\int_D G(t-s,x,y)\sigma_n(u_n(s,y))\dot{W}(dyds).\label{eq: mild u_n}
\end{align}
 Then we can define the unique local mild solution for \eqref{eq1}
\begin{align}
    u(t,x):=u_n(t,x) \textrm{ for all }x\in D \textrm{ and }t\in [0,\tau_n^\infty].\label{def:loc u}
\end{align}
where 
\begin{align}
    \tau_n^\infty:=\inf\left\{t>0:\sup_{x\in D}u_n(t,x)\geq n\right\}.\label{def:tau_n}
\end{align}
Define the explosion time
\begin{align}
    \tau_\infty^\infty &=\sup_n\tau_n^\infty. \label{def:tau_inf^inf}
\end{align}
A local mild solution explodes in finite time if $\tau_\infty^\infty<\infty$. A local mild solution is called a global mild solution if the solution never explodes with probability one, $\mathbb{P}(\tau_\infty^\infty<\infty)=1$.

Our proof is inspired by the approach in \cite{salins2025}, which leverages the quadratic variation of the spatial $L^1$ norm to establish the result for the critical case where the growth rate is $\frac{3}{2}$ in the one-dimensional heat equation. As in the one-dimensional case, the first step is to show that the spatial $L^1$ norm remains finite. The spatial $L^1$ norm of $u(t,x)$ is a local martingale when Neumann or periodic boundary conditions are imposed. Because the spatial integral $\int_D G(t,x,y)dx = 1$ for all $y \in D$, by integrating \eqref{eq: mild u_n} it follows that
\begin{equation}
    |u(t)|_{L^1(D)} := \int_D u(t,x)dx = I(t),
\end{equation}
where
\begin{align}
    I(t):=\int_D u(0,x)dx+\int_0^t\int_D\sigma(u(s,y))\dot{W}(dyds)\label{eq3}
\end{align}
which is a local martingale.

In the setting of Dirichlet boundary condition, the value of $\int_DG(t,x,y)dx$ will depend on $t$ and $x$. $|u(t)|_{L^1(D)}$ is not a semimartingale or even a local semimartingale. For this reason, we originally thought that the arguments of \cite{salins2025} were only applicable for Neumann or periodic boundary conditions. This article proves that we can extend \cite{Franzova1999} so $\gamma$ can reach $1+\frac{1-\alpha/2}{d}$ even in Dirichlet setting. Despite the fact that $|u(t)|_{L^1}$ is not a local martingale, we make the important observation that in the Dirichlet setting, $|u(t)|_{L^1(D)}\leq I(t)$ with probability 1 where $I(t)$ is the positive local martingale defined in \eqref{eq3}. Thus, even when Dirichlet boundary conditions are imposed, the spatial $L^1$ norm of $u(t,x)$ cannot explode in finite time and we have access to sensitive estimates based on the quadratic variation of $I(t)$. 

We define an $I$ stopping time for $M>0$:
\begin{align}
    \tau_M^I:=\inf\{t\in [0,\tau_\infty^\infty]:I(t)>M\}.\label{eq4}
\end{align}
where $I(t)$ is defined by $\eqref{eq3}$. Using Doob's submartingale inequality we can prove that for any $T>0$ the $L^1$ norm $\int_Du(t\land \tau_n^{\infty},x)dx$ cannot explode before $T$. The estimate is independent of $n$. Then we show that the quadratic variation of \eqref{eq3} is bounded in a way that is independent of $n$. It is proven in Lemma \ref{lem:quad} that
\begin{align}
        \mathbb{E}\int_0^{ \tau_M^I}\int_D\int_D\Lambda(x,y)\sigma(u(s,x))\sigma(u(s,y))dxdyds\leq M^2.\label{eq5}
\end{align}
Where $\Lambda$ is the covariance kernel. After that, we prove a similar $L^\infty$ bound on the stochastic convolution in Theorem \ref{thm:factorization} that we believe is novel and may be of independent interest. The proof is much more complicated than the space-time white noise case investigated in \cite{salins2025}.
\begin{theorem}\label{thm:factorization}
Let $p$ be large enough that $\frac{1+\beta}{p}<\frac{1-\eta}{2}-\frac{\beta}{p-2}$, where $\beta$ and $\eta$ are given in Assumption \ref{assumption3} in section 2 and assume that $\varphi(t,x)$ is adapted and bounded.
Define the stochastic convolution
\begin{align}
    Z^{\varphi}(t,x)=\int_0^t\int_D G(t-s,x,y)\varphi(s,y)\dot{W}(dyds).\label{eq33}
\end{align}
There exists $C_p>0$, independent of $T>0$ and $\varphi$ such that
    \begin{align}
     \begin{split}
          &\mathbb{E}\sup_{t\in [0,T]}\sup_{x\in D}|Z^{\varphi}(t,x)|^p\\&\leq  C_p T^{\frac{(1-\eta)(p-2)}{2}-2\beta} \mathbb{E}\int_0^T\int_D\int_D|\varphi(s,y_1)\varphi(s,y_2)|^{\frac{p}{2}}\Lambda(y_1,y_2)dy_1dy_2ds.
     \end{split}\label{eq34}
\end{align}
\end{theorem}
Furthermore, because
\begin{align}
    \begin{split}
        &\mathbb{E}\int_0^T\int_D\int_D|\varphi(s,y_1)\varphi(s,y_2)|^{\frac{p}{2}}\Lambda(y_1,y_2)dy_1dy_2ds\\
    \leq &\mathbb{E}\sup_{t\in[0,T]}|\varphi(t)|_{L^\infty(D)}^{p-2}\int_0^T\int_D\int_D|\varphi(s,y_1)\varphi(s,y_2)|\Lambda(y_1,y_2)dy_1dy_2ds,
    \end{split}
\end{align}
we can bound the stochastic convolution by \eqref{eq5}. The proof of Theorem \ref{thm:factorization} is in section 5.
Finally we can show solution will not blow up in finite time by using a Borel-Cantelli argument to show the $L^\infty$ norm can double only a finite number of times.

In section 2 we will present the assumptions and the main result for this paper. In section 3 there will be an explicit example. In section 4 we will show the $L^1$ norm of $u(t,x)$ will not explode in finite time. In section 5 we will show the stochastic convolution is bounded by the quadratic variation. Finally in section 6 we will prove the $L^\infty$ norm will not explode.
\section{Assumptions and Results}
Let $L^p:=L^p(D),p\geq 1$ denote the standard $L^p$ spaces on $D$ where $D$ is given in Assumption \ref{assumption1} endowed with the norms
\begin{align}
    |\varphi|_{L^p}&:=\Big(\int_D |\varphi(y)|^pdy\Big)^{\frac{1}{p}},p\in [1,\infty),\label{def:Lp}\\
    |\varphi|_{L^\infty}&:=\sup_{x\in D}|\varphi(x)|.\label{def:L infty}
\end{align}
\begin{assumption}\label{assumption1}
    $D\subset\mathbb{R}^d$ is an open, connected and bounded domain. $A$ is a second-order self-adjoint elliptic operator, which means
    \begin{align}
        Au(x)=\sum_{i=1}^d\sum_{j=1}^d\frac{\partial}{\partial x_i}(a_{ij}(x)\frac{\partial}{\partial x_j}u(x)), \label{def:diff operator}
    \end{align}
    for some symmetric $a_{ij}\in C^2(\overline{D})$ that satisfy uniformly elliptic condition
    \begin{align}
        a_{ij}=a_{ji},\sum_{i,j=1}^da_{ij}(x)\xi_i\xi_j\geq \theta|\xi|^2, \label{assum:unif ellip}
    \end{align}
    for some $\theta>0$, and all $x\in D$ and all $\xi\in \mathbb{R}^d$. We assume either:
    \begin{itemize}
        \item [(a)] Periodic boundary condition,
        \item [(b)] Neumann boundary condition:
        \begin{align}
             Bu(t,x)=\nabla u(t,x)\cdot(a(x)\vec{n}(x))=0,x\in\partial D,\label{Boundary:Nemann}
        \end{align}
         where $\vec{n}(x)$ is the outer pointing normal vector,
     \item [(c)] or Dirichlet boundary condition:
         \begin{align}
             u(t,x)=0, x\in \partial D.
         \end{align}
    \end{itemize}
   Because $A$ is self-adjoint, there exist eigenvalues/eigenfunctions of the realization of $A$ in $L^2(D)$, with the imposed boundary conditions, are given by
\begin{align}
    A e_k=-\alpha_k e_k, \ \ \ \  0\leq \alpha_1\leq \alpha_2\leq \cdots\leq\alpha_k\leq \cdots.\label{eq6}
\end{align}
The corresponding fundamental solution is denoted by $G(t,x,y)$ where
\begin{align}
    G(t,x,y)=\sum_{k=1}^\infty e^{-\alpha_k t}e_k(x)e_k(y).\label{eq:heat kernel}
\end{align}
The fundamental solution has the property that for any $\phi \in L^2(D)$, we have
$u(t,x) = \int_D G(t,x,y)\phi(y)dy$
solves the heat equation $\frac{\partial u}{\partial t}(t,x) = A u(t,x)$ with initial value $u(0,x) = \phi(x)$.
\end{assumption}

\begin{definition}
  The colored noise is a Gaussian with the formal covariance structure
\begin{align}
    \mathbb{E}[\dot{W}(t,x)\dot{W}(s,y)]=\delta_0(t-s)\Lambda(x,y).\label{eq8}
\end{align}
Where $\delta_0$ is the Dirac delta measure and $\Lambda(x,y)$ is a positive definite function which describes the correlation. If it is space-time white noise then $\Lambda(x,y)=\delta_0(x-y)$. More rigorously, for any adapted $\varphi,\psi \in C_0^\infty([0,T]\times D)$ test functions, 
\begin{align}
     \begin{split}
         &\mathbb{E}\Big(\int_0^T \int_D \varphi(s,y)W(dyds)\Big)\Big(\int_0^T \int_D \psi(s,y)W(dyds)\Big)\\
    = & \mathbb{E}\int_0^T\int_D\int_D\varphi(s,y_1)\psi(s,y_2)\Lambda(x,y)dy_1dy_2ds.
     \end{split}
\end{align}
\end{definition}
\begin{assumption}\label{assumption2}
    The covariance kernel $\Lambda(x,y)$ is positive and positive definite.
\end{assumption}
We assume the following on the fundamental solution and the noise covariance kernel, see \cite{Cerrai2003, DaPrato_Zabczyk_2014}.

\begin{assumption}\label{assumption3}
    There exists $C>0, \beta>0$ and $\eta\in (0,1)$ s.t.
    \begin{itemize}
        \item [(A)]$\sup_{x,y\in D}G(t,x,y)\leq Ct^{-\beta}$,
        \item [(B)]$\sup_{x\in D}\int_D\int_D G(t,x,y_1)G(t,x,y_2)\Lambda(y_1,y_2)dy_1dy_2\leq Ct^{-\eta}$,
        \item [(C)]$\int_D\int_D\Lambda(y_1,y_2)dy_1dy_2<\infty$.
    \end{itemize}
\end{assumption}
One can check the local mild solution of SPDE is function-valued under these conditions and the point $B$ comes from Appendix $B$ in \cite{Franzova1996}.

\begin{assumption}\label{assumption4}
    $\sigma:\mathbb{R} \to \mathbb{R}$ is locally Lipschitz continuous, $\sigma(0)=0$ and
    \begin{align}
        |\sigma(u)|\leq C(1+|u|^{1+\frac{1-\eta}{2\beta}}),\label{eq9}
    \end{align}
    where $\eta$ and $\beta$ are from Assumption \ref{assumption3}, WLOG we can assume $\sigma\geq 0$.
\end{assumption}
In Mueller's setting \cite{Mueller1991,Muller1998,Mueller2000}, $\beta=\frac{1}{2}$ and $\eta=\frac{1}{2}$ leading to the same $|u|^{\frac{3}{2}}$ condition. Additionally, our result extends Mueller's results to Dirichlet and Neumann settings.
We show in section 3 when the color noise uses a spatially homogeneous Riesz kernel restricted on a bounded domain with paramater $\alpha$ as is shown in \cite{Franzova1999}, $1+\frac{1-\eta}{2\beta}$ becomes $1+\frac{1-\alpha/2}{d}$. 
\begin{assumption}\label{assumption5}
    \begin{align}
        u(0,x)\geq 0 \text{ for a.e. } x\in D. \label{assum: positivity}
    \end{align}
\end{assumption}
Assumption \ref{assumption5} combing with $\sigma(0)=0$ guarantees that the solution remains positive due to the comparison principle\cite{Kotelenez1992,Mueller199101}.

\noindent Theorem \ref{thm:nonexplosion} proves that when assumptions \ref{assumption1}-\ref{assumption5} are satisfied, the mild solution is global.
\begin{theorem}\label{thm:nonexplosion}
    Assume Assumption \ref{assumption1}-\ref{assumption5} and the initial data $x\mapsto u(0,x)\in L^{\infty}(D)$, then there exists a unique global mild solution to \eqref{eq1}.
\end{theorem}

\section{Examples}

In this section, we describe a large class of examples that satisfy Assumptions \ref{assumption1}--\ref{assumption4}. First we show that our theory applies to \cite{Franzova1999} and improves it from $\gamma<1+\frac{1-\eta}{2\beta}$ to $\gamma\leq 1+\frac{1-\eta}{2\beta}$. Let $D=[0,\pi]^d$ be a rectangular domain in $\mathbb{R}^d$ for some $d\geq 1$. $A=\Delta$ is the  Laplacian. Thus, it is obvious that $A$ and $D$ for this example satisfies Assumption \ref{assumption1}. Use the spatially homogeneous Riesz-kernel noise considered in \cite{Franzova1999} where
\begin{align*}
    \Lambda(x,y)=|x-y|^{-\alpha}\mathbbm{1}_{\{x,y\in D\}},\quad 0<\alpha<\min\{2,d/2\}.
\end{align*}
$\alpha$ indicates the size of correlation and from Appendix $B$ in \cite{Franzova1996} one can check that $\eta=\frac{\alpha}{2}$ which satisfies Assumption \ref{assumption3}$(B)$. This kernel satisfies Assumption \ref{assumption2} by construction. Also, in this setting $\beta=\frac{d}{2}$ for Assumption \ref{assumption3}$(A)$ and because $\alpha<\min\{2,d/2\}$ Assumption \ref{assumption3}$(C)$ and $\eta\in (0,1)$ are satisfied. Here $\sigma(u)=u^{\gamma}$ for $\gamma\leq 1+\frac{1-\eta}{2\beta}=1+\frac{1-\alpha/2}{d}$. From Theorem \ref{thm:nonexplosion} it follows that we extend the result in \cite{Franzova1999} by allowing $\gamma$ to reach the level $1+\frac{1-\alpha/2}{d}$.

We could alternatively use the spectral noise considered in section 5 of Da Prato and Zabczyk's book \cite{DaPrato_Zabczyk_2014} where
\begin{align}
    \Lambda(x,y)=\sum_{k=1}^\infty\lambda_k^2 e_k(x)e_k(y). \label{eq:Cov kernel spec}
\end{align}
In our setting, we additionally require $\Lambda(x,y)$ to be positive pointwise to satisfy Assumption \ref{assumption2}. We can always construct examples of positive $\Lambda$ satisfying \eqref{eq:Cov kernel spec} by defining
\begin{align}
    \begin{split}
        \Lambda(x,y)&=\int_0^\infty s^{\theta-1}e^{-as}G(s,x,y)ds,
    \end{split}\label{eq:Lam def spec}
\end{align}
for any $\theta>0,a>0$ where $G(t,x,y)$ is the fundamental solution. Observe $\Lambda(x,y)\geq 0$ because $G(t,x,y)\geq 0$. Furthermore, $\Lambda$ can be written as \eqref{eq:Cov kernel spec}
\begin{align}
    \begin{split}
        \Lambda(x,y)&=\sum_{k=1}^\infty\Big(\int_0^\infty s^{\theta-1}e^{-(\alpha_k+a) s}ds\Big)e_k(x)e_k(y)\\
        &=\Gamma(\theta)\sum_{k=1}^\infty (a+\alpha_k)^{-\theta}e_k(x)e_k(y), 
    \end{split}\label{eq:Lam def spec2}
\end{align}
where $\Gamma(\theta)$ is the Gamma function.
If Dirichlet boundary condition is imposed, then $\alpha_1>0$ and we can choose $a=0$. Thus, $\lambda_k^2=\Gamma(\theta)(a+\alpha_k)^{-\theta}$.

The quantities from Assumption \ref{assumption3} are bounded by
\begin{enumerate}
    \item[(A)] $\sup_{x \in D} \sup_{y \in D}  \displaystyle {G(t,x,y) \leq  \sum_{k=1}^\infty e^{-\alpha_k t} |e_k|_{L^\infty}^2}$.
    \item[(B)] $\displaystyle{\sup_{x \in D} \int_D \int_D G(t,x,y_1)G(t,x,y_2)\Lambda(y_1,y_2)dy_1dy_2 = \sum_{k=1}^\infty \Gamma(\theta)(a + \alpha_k)^{-\theta }e^{-2\alpha_k t} |e_k(x)|^2}$ 
    \item[(C)] $\displaystyle{\int_D \Lambda(y_1,y_2)dy_1dy_2 = \sum_{k=1}^\infty (a + \alpha_k)^{-\theta} \left<1,e_k\right>_{L^2}^2\leq C \sup_{k} (a + \alpha_k)^{-\theta}}$
\end{enumerate}

In the setting where $A$ is the Laplacian and $D$ is a rectangular domain with periodic, Dirichlet, or Neumann boundary conditions, 
\[\alpha_k \sim k^{\frac{2}{d}} \text{ and } \sup_k |e_k|_{L^\infty} <\infty.\]
In these settings 
\[\sum_{k=1}^\infty \exp\left({-tk^{\frac{2}{d}} }\right) \approx \int_1^\infty \exp\left({-tx^{\frac{2}{d}}}\right)dx = Ct^{-\frac{d}{2}},  \]
and
\[\sum_{k=1}^\infty k^{-\frac{2\theta}{d}}\exp\left(-tk^{\frac{2}{d}}\right) \approx \int_1^\infty x^{-\frac{2\theta}{d}}\exp\left(-t x^{\frac{2}{d}}\right)dx = C t^{\theta - \frac{d}{2} }. \]
Therefore, $\beta = \frac{d}{2}$, $\eta = \max\{\frac{d}{2} - \theta,0\}$, and we require $\theta> \frac{d}{2}-1$ .

Of course, our theory is applicable to any other setting where $\eta$ and $\beta$ can be computed.
\section{The $L^1$ norm of $u(t,x)$}
\begin{proposition}\label{prop:positivity}
Assume Assumption \ref{assumption1}-\ref{assumption5}, let $u(t,x)$ be the local mild solution to \eqref{eq1} then
    \begin{align}
        u(t,x)\geq 0 \text{ for all }t\geq 0\text{ and a.e. }x\in D.\label{eq:positivity}
    \end{align}
\end{proposition}
This is the result of the comparison principle\cite{Kotelenez1992,Mueller199101} as mentioned before. Let $\tau_n^\infty$ be defined in \eqref{def:tau_n}  and define 
\begin{align}
    I_n(t):=\int_D u(0,x)dx+\int_0^{t\land\tau_n^\infty}\int_D\sigma(u(s,x))\dot{W}(dxds).\label{eq15}
\end{align}
Also, we recall several properties from fundamental solution
\begin{itemize}
    \item $G(t,x,y)\geq 0 \text{ for all } t>0,x,y\in D$ which can be inferred by comparison principle.
   \item $\int_D G(t,x,y)dy\leq 1$ which is the result of conservation of total heat.
\end{itemize}
Then the following lemma gives an upper bound for $L^1$ norm of $u(t,x)$.
\begin{lemma}\label{lem:I_n bound}
For each $n$,
    \begin{align}
        \int_D u(t\land \tau_n^\infty, x)dx\leq I_n(t),\label{eq16}
    \end{align}
where $I_n$ is defined in \eqref{eq15}.
\end{lemma}
\begin{proof}
We prove \eqref{eq16} case by case. If $D$ has Neumann or periodic boundary conditions, then it follows by Assumption \ref{assumption1} that for any $t>0,y\in D$
\begin{align}
    \int_D G(t,x,y)dx=1.\label{eq17}
\end{align}
Thus, it follows by integrating \eqref{eq: mild u_n} that
\begin{align}
    \int_D u(t\land \tau_n^\infty, x)dx=I_n(t) \text{ for all } t>0,n\in\mathbb{N}.\label{eq18}
\end{align}
If Dirichlet boundary conditions are imposed, then recall that Theorem 5.4 in \cite{DaPrato_Zabczyk_2014} shows that under uniformly ellipticity and Lipschitz condition a mild solution is also a weak solution, meaning that for every $g\in C_0^2(D)$ where $C_0^2(D)$ is defined by
\begin{align}
    C_0^2(D):=\Big\{g: g \text{ is continuous and twice-differentiable with }g(x)=0 \text{ for all }x\in\partial D\Big\}.
\end{align}
It follows that for each $n$
\begin{align}
  \begin{split}
        \int_D u(t\land \tau_n^\infty,x)g(x)dx=&\int_D u(0,x)g(x)dx+\int_0^t\int_D u(s\land \tau_n^\infty,x)A g(x)dxds\\
    &+\int_0^t\int_D \sigma(u(s\land \tau_n^\infty,x))g(x)\dot{W}(dxds).
  \end{split}\label{eq19}
\end{align}
Define a sequence of functions $\{g_m(x)\}_{m=1}^\infty$ where $g_m(x)$ is defined by
\begin{align}
    g_m(x)=\int_D G(2^{-m},x,y)dy.\label{eq20}
\end{align}
Thus, $g_m(x)$ is the solution to the parabolic equation $\frac{\partial v}{\partial t} = A v(t,x)$ with initial data $v(0,x) \equiv 1$, evaluated at $t=2^{-m}$. $g_m(x)$ have the properties that
\begin{itemize}
    \item $0 \leq g_m(x) \leq 1$ for all $x \in D$ by the properties of fundamental solution.
    \item $g_m\in C_0^2(D)$ because $G(t,x,y)$ is smooth and we can differentiate under integral.
    \item Let $\{X_x(t)\}_{x\in D,t\geq 0}$ be the stochastic process associated with the stochastic differential equation with generator $A$. Uniformly elliptic condition guarantees the existence of $X_x(t)$. Thus, let
    \begin{align}
        \tau_{D,x}:=\inf\{t>0:X_x(t)\notin D\}.\label{def:tau_D}
    \end{align}
    Then it follows that
    \begin{align}
       \int_D G(t,x,y)dy=\mathbb{P}(\tau_{D,x}>t).
    \end{align}
    Thus, $\int_D G(t,x,y)dy$ is decreasing in $t$. It follows that
    \begin{align}
        Ag_m(x)=\int_D AG(2^{-m},x,y)dy=\int_D\partial_t G(2^{-m},x,y)dy=\partial_t\int_DG(2^{-m},x,y)dy\leq 0.
    \end{align}
    Thus, $Ag_m(x)\leq 0$ for all $x\in D$.
   
    \item For each $x \in D$, $\lim_{m \to \infty} g_m(x) = 1$ by continuity of solution of the parabolic equation.
\end{itemize}

Because $u(t,x)$ is a local mild solution, 
\begin{align}
\begin{split}
    \int_D u(t\land \tau_n^\infty,x)g_m(x)dx=&\int_D u(0,x)g_m(x)dx+\int_0^t\int_D u(s\land \tau_n^\infty,x)A g_m(x)dxds\\
    &+\int_0^t\int_D \sigma(u(s\land \tau_n^\infty,x))g_m(x)\dot{W}(dxds).
\end{split}\label{eq21}
\end{align}
Thus, because $Ag_m(x)\leq 0$ and \eqref{eq:positivity},
\begin{align}
    \int_Du(t\land \tau_n^\infty,x)g_m(x)dx\leq \int_D u(0,x)g_m(x)dx+\int_0^t\int_D \sigma(u(s\land \tau_n^\infty,x))g_m(x)\dot{W}(dxds).\label{eq22}
\end{align}
Observe that the definition of $\tau_n^\infty$ guarantees that $u(t\land \tau_n^\infty,x),u(0,x)$ and $g_m(x)$ are bounded and also $D$ is a bounded domain. By dominated convergence theorem we can take the limit as $m\rightarrow \infty$ then it follows that with probability 1
\begin{align}
\begin{split}
    \lim_{m\rightarrow\infty}\int_Du(t\land \tau_n^\infty,x)g_m(x)dx &=\int_Du(t\land \tau_n^\infty,x)dx,\\
    \lim_{m\rightarrow\infty}\int_D u(0,x)g_m(x)dx &=\int_D u(0,x)dx.
\end{split}\label{eq23}
\end{align}
For martingale part, by taking $L^2$ norm then by Assumption \ref{assumption3}$(C)$ and boundness of $u$ and $g_m(x)$ it follows that
\begin{align}
\begin{split}
    &\mathbb{E}|\int_0^t\int_D \sigma(u(s\land \tau_n^\infty,x))(g_m(x)-1)\dot{W}(dxds)|^2\\
    =&\mathbb{E}\int_0^t\int_D\int_D  \sigma(u(s\land \tau_n^\infty,x))\sigma(u(s\land \tau_n^\infty,y))(g_m(x)-1)(g_m(y)-1)\Lambda(x,y) dxdyds.
\end{split}\label{eq24}
\end{align}
Furthermore, by Assumption \ref{assumption4} it follows that $|\sigma(u(s\land\tau_n^\infty,x))|\leq C(1+|u(s\land\tau_n^\infty)|^{1+\frac{1-\eta}{2\beta}})\leq C(1+n^{\frac{1-\eta}{2\beta}})$ and by Assumption \ref{assumption3}(C) we have $\int_D\int_D\Lambda(x,y)dxdy<\infty$. Combine with $0\leq g_m\leq 1$ for all $x\in D$ and $m$, it follows that
\begin{align}
    \mathbb{E}|\int_0^t\int_D \sigma(u(s\land \tau_n^\infty,x))(g_m(x)-1)\dot{W}(dxds)|^2\leq C_n\int_D\int_D\Lambda(x,y)dxdy.\label{eq24*}
\end{align}
Thus, by dominated convergence theorem it follows that $\int_0^t\int_D \sigma(u(s\land \tau_n^\infty,x))g_m(x)\dot{W}(dxds)$ converges to $\int_0^t\int_D \sigma(u(s\land \tau_n^\infty,x))\dot{W}(dxds)$ in $L^2(D)$. Thus, $\eqref{eq16}$ is proven by taking limit of \eqref{eq21}.
\end{proof}

Define $I(t)$ such that 
\begin{align}
    I(t):=I_n(t),\textrm{ for }t\in [0,\tau_n^\infty].\label{eq25}
\end{align}
It follows that $I(t)$ is a nonnegative local martingale. Define the $I$ stopping times for $M>0$
\begin{align}
    \tau_M^I:=\inf\{t\in [0,\tau_\infty^\infty]:I(t)>M\}.\label{eq26}
\end{align}
\begin{lemma}\label{lem:quad}
    For any $T>0$ and $M>0$,
    \begin{align}
        \mathbb{P}\Big(\sup_{t\in[0, T\land\tau_\infty^\infty]}|u(t)|_{L^1}>M\Big)\leq \frac{|u(0)|_{L^1}}{M}.\label{eq27}
    \end{align}
    In particular,
    \begin{align}
         \mathbb{P}\Big(\sup_{t\in[0, T\land\tau_\infty^\infty]}|u(t)|_{L^1}<\infty\Big)=1.\label{eq28}
    \end{align}
    Furthermore, for any $M>0$, the quadratic variation of $I(t)$ satisfies 
    \begin{align}
        \mathbb{E}\int_0^{ \tau_M^I}\int_D\int_D\Lambda(x,y)\sigma(u(s,x))\sigma(u(s,y))dxdyds\leq M^2.\label{eq29}
\end{align}
\end{lemma}
\begin{proof}
    By \eqref{eq16} and Doob's inequality
    \begin{align}
        \mathbb{P}\Big(\sup_{t\in [0,T\land\tau_n^{\infty}]}|u(t)|_{L^1}>M\Big)\leq \mathbb{P}\Big(\sup_{t\in [0,T]}I_n(t)>M\Big)\leq \frac{I_n(0)}{M}=\frac{|u(0)|_{L^1}}{M}.\label{eq30}
    \end{align}
    This bound does not depend on $n$. Therefore, \eqref{eq27} holds. Take $M\uparrow\infty$ then \eqref{eq28} holds. Then apply It\^{o} formula to \eqref{eq15}. For any $M>0,n>0$
    \begin{align}
        \begin{split}
            \mathbb{E}(I(t\land \tau_M^I))^2&=\mathbb{E}(I_n(t\land \tau_M^I))^2\\
            &=\mathbb{E}(I_n(0))^2+\mathbb{E}\int_0^{ t\land\tau_n^\infty\land\tau_M^I}\int_D\int_D\Lambda(x,y)\sigma(u(s,x))\sigma(u(s,y))dxdyds.
        \end{split}\label{eq31}
    \end{align}
    Therefore,
    \begin{align}
        \mathbb{E}\int_0^{ t\land\tau_n^\infty\land\tau_M^I}\int_D\int_D\Lambda(x,y)\sigma(u(s,x))\sigma(u(s,y))dxdyds\leq M^2.\label{eq32}
    \end{align}
    This bound does not depend on $n$ or $t$. Thus, \eqref{eq29} holds.
\end{proof}
\section{Moment estimates on the stochastic convolution}
\noindent\textit{Proof of Theorem \ref{thm:factorization}}. Let $p$ be large enought such that $\frac{1+\beta}{p}<\frac{1-\eta}{2}-\frac{\beta}{p-2}$, where $\beta$ and $\eta$ are given in Assumption \ref{assumption3} in section 2 and assume that $\varphi(t,x)$ is adapted and bounded.
Define the stochastic convolution
\begin{align}
    Z^{\varphi}(t,x)=\int_0^t\int_D G(t-s,x,y)\varphi(s,y)\dot{W}(dyds).
\end{align}
Use Da Prato and Zabczyk's factorization method in Theorem 5.10\cite{DaPrato_Zabczyk_2014}, choose $\alpha\in(\frac{1+\beta}{p},\frac{1-\eta}{2}-\frac{\beta}{p-2})$. Define
\begin{align}
    Z_{\alpha}^{\varphi}(t,x)=\int_0^t \int_D(t-s)^{-\alpha}G(t-s,x,y)\varphi(s,y)\dot{W}(dyds).\label{eq35}
\end{align}
Then by the factorization lemma
\begin{align}
     Z^{\varphi}(t,x)=C\int_0^t\int_D (t-s)^{\alpha-1}G(t-s,x,y)Z_{\alpha}^{\varphi}(s,y)dyds.\label{eq36}
\end{align}
Now using H$\ddot{\textrm{o}}$lder's inequality we have 
\begin{align}
    \begin{split}
        \sup_{t\in [0,T]}\sup_{x\in D}|Z^\varphi(t,x)|\leq& \sup_{t\in [0,T]}\sup_{x\in D}C \Big(\int_0^t\int_D(t-s)^{\frac{(\alpha-1)p}{p-1}}G^{\frac{p}{p-1}}(t-s,x,y)dyds\Big)^{\frac{p-1}{p}}\\
        &\times \Big(\int_0^t\int_D |Z_\alpha^\varphi(s,y)|^pdyds\Big)^{\frac{1}{p}}.
    \end{split}\label{eq38}
\end{align}
Then by Assumption \ref{assumption3}(A) and $0\leq \int_DG(t,x,y)dy\leq 1$ it follows that
\begin{align}
    \begin{split}
        \int_0^t\int_D(t-s)^{\frac{(\alpha-1)p}{p-1}}G^{\frac{p}{p-1}}(t-s,x,y)dyds&\leq \int_0^t (t-s)^{\frac{(\alpha-1)p-\beta}{p-1}}\int_DG(t-s,x,y)dyds\\
        &\leq \int_0^t (t-s)^{\frac{(\alpha-1)p-\beta}{p-1}}ds.
    \end{split}\label{eq38*}
\end{align}
Since we choose $p,\alpha$ such that $\alpha p-\beta-1>0$, $(t-s)^{\frac{(\alpha-1)p-\beta}{p-1}}$ is integrable between $0$ and $t$. Thus, this yields
\begin{align}
    \mathbb{E}\sup_{t\in [0,T]}\sup_{x\in D}|Z^{\varphi}(t,x)|^p\leq C T^{\alpha p-\beta-1}\mathbb{E}\int_0^T\int_D |Z_\alpha^\varphi(t,x)|^pdxdt.\label{eq39}
\end{align}
Using BDG inequality then it follows that
\begin{align}
    \begin{split}
        \mathbb{E}|Z_{\alpha}^{\varphi}(t,x)|^p\leq C_p\mathbb{E}\Big(&\int_0^t (t-s)^{-2\alpha}\int_D\int_D G(t-s,x,y_1)G(t-s,x,y_2)\\
    &\times\varphi(s,y_1)\varphi(s,y_2)\Lambda(y_1,y_2)dy_1dy_2ds\Big)^{\frac{p}{2}}.
    \end{split}\label{eq40}
\end{align}
Thus, 
\begin{align}
    \begin{split}
        &\mathbb{E}\int_0^T\int_D |Z_\alpha^\varphi(t,x)|^pdxdt\\
        \leq & C_p\mathbb{E}\int_0^T\int_D\Big(\int_0^t (t-s)^{-2\alpha}\int_D\int_D G(t-s,x,y_1)G(t-s,x,y_2)\\
        &\times\varphi(s,y_1)\varphi(s,y_2)\Lambda(y_1,y_2)dy_1dy_2ds\Big)^{\frac{p}{2}}dxdt.
    \end{split}\label{eq41}
\end{align}
By a duality argument it follows that with probability 1
\begin{align}
    \begin{split}
        &\int_0^T\int_D\Big(\int_0^t (t-s)^{-2\alpha}\int_D\int_D G(t-s,x,y_1)G(t-s,x,y_2)\\
        &\times\varphi(s,y_1)\varphi(s,y_2)\Lambda(y_1,y_2)dy_1dy_2ds\Big)^{\frac{p}{2}}dxdt\\
        =&\sup_{|h|_{L^{\frac{p}{p-2}}}=1}\\
        &\Big(\int_0^T\int_D\int_0^t\int_D\int_Dh(t,x)(t-s)^{-2\alpha}G(t-s,x,y_1)G(t-s,x,y_2)\\
        &\times\varphi(s,y_1)\varphi(s,y_2)\Lambda(y_1,y_2)dy_1dy_2dsdxdt\Big)^{\frac{p}{2}}.
    \end{split}\label{eq41*}
\end{align}
By Assumption \ref{assumption2} and positivity of $G(t,x,y)$ and $\Lambda(x,y)$ we can treat $(t-s)^{-2\alpha}G(t-s,x,y_2)\Lambda(y_1,y_2)dy_1dy_2dsdxdt$ as a measure then use H$\ddot{\textrm{o}}$lder's inequality with parameters $\frac{p}{p-2}$ and $\frac{p}{2}$ to show that for any $h$ such that $|h|_{L^{\frac{p}{p-2}}}=1$ we have with probability 1
\begin{align}
   \begin{split}
        &\Big(\int_0^T\int_D\int_0^t\int_D\int_Dh(t,x)(t-s)^{-2\alpha}G(t-s,x,y_1)G(t-s,x,y_2)\\
        &\times\varphi(s,y_1)\varphi(s,y_2)\Lambda(y_1,y_2)dy_1dy_2dsdxdt\Big)^{\frac{p}{2}}\\
    \leq &\Big(\int_0^T\int_D\int_0^t\int_D\int_D|h(t,x)|^{\frac{p}{p-2}}(t-s)^{-2\alpha}G^{\frac{p}{p-2}}(t-s,x,y_1)G(t-s,x,y_2)\\
    &\times\Lambda(y_1,y_2)dy_1dy_2dsdxdt\Big)^{\frac{p-2}{2}}\\
    &\times\Big(\int_0^T\int_D\int_0^t\int_D\int_D (t-s)^{-2\alpha}|\varphi(s,y_1)\varphi(s,y_2)|^{\frac{p}{2}}G(t-s,x,y_2)\Lambda(y_1,y_2)dy_1dy_2dsdxdt\Big)\\
    &=:A^{\frac{p-2}{2}}B.
   \end{split}\label{eq41**}
\end{align}
Now by Assumption \ref{assumption3}(A), $G^{\frac{p}{p-2}}(t-s,x,y_1)=G^{\frac{2}{p-2}}(t-s,x,y_1)G(t-s,x,y_1)\leq C(t-s)^{-\frac{2\beta}{p-2}}G(t-s,x,y_1)$. It follows that with probability 1
\begin{align}
    \begin{split}
    A \leq \int_0^T\int_D\int_0^t\int_D\int_D &|h(t,x)|^{\frac{p}{p-2}}(t-s)^{-2\alpha-\frac{2\beta}{p-2}}G(t-s,x,y_1)\\
    &\times G(t-s,x,y_2)\Lambda(y_1,y_2)dy_1dy_2dsdxdt.
    \end{split}\label{eq42}
\end{align}
By Assumption \ref{assumption3}(B) and \eqref{eq42},
\begin{align}
    A\leq C\int_0^T\int_D\int_0^t|h(t,x)|^{\frac{p}{p-2}}(t-s)^{-2\alpha-\frac{2\beta}{p-2}-\eta}dsdxdt\text{ with probability 1}.
\end{align}
Since we chose $\alpha,p$ such that $2\alpha+\frac{2\beta}{p-2}+\eta<1$ and $|h|_{L^{\frac{p}{p-2}}}=1$, it follows that with probability 1
\begin{align}
    \begin{split}
       A \leq & C T^{(1-\eta)-2\alpha-\frac{2\beta}{p-2}}\int_0^T\int_D |h(t,x)|^{\frac{p}{p-2}}dxdt\leq CT^{(1-\eta)-2\alpha-\frac{2\beta}{p-2}}.
    \end{split}\label{eq43}
\end{align}
Then by the property that $0\leq \int_DG(t-s,x,y_2)dx\leq 1$
\begin{align}
    \begin{split}
  B\leq &\int_0^T\int_0^t\int_D\int_D(t-s)^{-2\alpha}|\varphi(s,y_1)\varphi(s,y_2)|^{\frac{p}{2}}\Lambda(y_1,y_2)dy_1dy_2dsdt\\
  =&\int_0^T\int_s^T\int_D\int_D(t-s)^{-2\alpha}|\varphi(s,y_1)\varphi(s,y_2)|^{\frac{p}{2}}\Lambda(y_1,y_2)dy_1dy_2dtds\\
  =&C\int_0^T\int_D\int_D(T-s)^{1-2\alpha}|\varphi(s,y_1)\varphi(s,y_2)|^{\frac{p}{2}}\Lambda(y_1,y_2)dy_1dy_2ds\\
  \leq & C T^{1-2\alpha}\int_0^T\int_D\int_D|\varphi(s,y_1)\varphi(s,y_2)|^{\frac{p}{2}}\Lambda(y_1,y_2)dy_1dy_2ds\text{ with probability 1.}
    \end{split}\label{eq44}
\end{align}
Thus, combining \eqref{eq43} and \eqref{eq44} it follows that
\begin{align}
    \begin{split}
        &\Big(\int_0^T\int_D\int_0^t\int_D\int_Dh(t,x)(t-s)^{-2\alpha}G(t-s,x,y_1)G(t-s,x,y_2)\\
        &\times\varphi(s,y_1)\varphi(s,y_2)\Lambda(y_1,y_2)dy_1dy_2dsdxdt\Big)^{\frac{p}{2}}\\
        \leq & A^{\frac{p-2}{2}}B\\
        \leq & CT^{\frac{(1-\eta)(p-2)}{2}-\beta-\alpha p+1}\int_0^T\int_D\int_D|\varphi(s,y_1)\varphi(s,y_2)|^{\frac{p}{2}}\Lambda(y_1,y_2)dy_1dy_2ds\text{ with probability 1.}
    \end{split}\label{eq45}
\end{align}
Then by \eqref{eq39} and \eqref{eq45} it follows that
\begin{align}
   \begin{split}
          \mathbb{E}\sup_{t\in [0,T]}\sup_{x\in D}|Z^{\varphi}(t,x)|^p\leq &C_p T^{\frac{(1-\eta)(p-2)}{2}-2\beta}\mathbb{E}
\int_0^T\int_D\int_D|\varphi(s,y_1)\varphi(s,y_2)|^{\frac{p}{2}}\Lambda(y_1,y_2)dy_1dy_2ds.
     \end{split}\label{eq45*}
\end{align}
Thus, \eqref{eq34} holds.

\section{Non-explosion of $u(t,x)$}
Let $M>0$ be arbitrary. In this section we will show that $u(t,x)$ cannot explode before time $\tau_M^I$. Then we can take the limit as $M\rightarrow \infty$ to prove that explosion cannot ever occur.

Fix $M>0$ and define a sequence of stopping time $\rho_n$. These stopping times depend on the choice of $M$.
\begin{align}
    \rho_0=\inf\{t\in [0,\tau_M^I]:|u(t)|_{L^\infty}=2^m\text{ for some }m\in\{1,2,3,\dots\}\}.\label{eq46}
\end{align}
$\rho_0=\tau_M^I$ if there is no doubling or halving before $\tau_M^I$. If $|u(\rho_n)|_{L^\infty}=2^m$ for some $n$ and $m\geq 2$ we define
\begin{align}
    \rho_{n+1}=\inf\Big\{t\in [\rho_n,\tau_M^I]:|u(t)|_{L^\infty}\geq 2^{m+1} \text{ or } |u(t)|_{L^\infty}\leq 2^{m-1}.\Big\},\label{eq47}
\end{align}
and if $|u(\rho_n)|_{L^\infty}=2$ then
\begin{align}
    \rho_{n+1}=\inf\Big\{t\in [\rho_n,\tau_M^I]:|u(t)|_{L^\infty}\geq 2^2.\Big\},\label{eq48}
\end{align}
$\rho_n=\tau_M^I$ if process stops doubling or halving. This basically follows the same idea as in \cite{salins2025}. The following lemma shows that the doubling probability is bounded by the quadratic variation of $I(t)$.
\begin{lemma}\label{lem:double u}
    For any $p$ large enough such that $\frac{1+\beta}{p}<\frac{1-\eta}{2}-\frac{\beta}{p-2}$, there exists a nonrandom constant $C_p>0$ and for any $M>0$ there exists a nonrandom constant $m_0=m_0(M)>0$ such that for any $n\in \mathbb{N}$, and $m>m_0$
    \begin{align}
\begin{split}
&\mathbb{P}\Big(|u(\rho_{n+1})|_{L^\infty}=2^{m+1}\Big||u(\rho_{n})|_{L^\infty}=2^{m}\Big)\\
&\leq C_p M^{\frac{(1-\eta)(p-2)}{2\beta}-2}\mathbb{E}\Big(\int_{\rho_n}^{\rho_{n+1}}\int_D\int_D\sigma(u(s,y_1))\sigma(u(s,y_2))\Lambda(y_1,y_2)dy_1dy_2ds\Big| |u(\rho_{n})|_{L^\infty}=2^{m}\Big).
\end{split}\label{eq49}
    \end{align}
    Importantly, the constant $C_p$ is independent of $m>m_0$ and so is the exponent of $M$.
\end{lemma}
\begin{proof}
    Let $M>0$ and assume that $2^m=|u(\rho_n)|_{L^\infty}$. By the semigroup property, the mild solution for $t\in [0,\rho_{n+1}-\rho_n]$, satisfies
    \begin{align}
        \begin{split}
            u(t+\rho_n,x)=&\int_D G(t,x,y)u(\rho_n,y)dy\\
            &+\int_0^t\int_DG(t,x,y)\sigma(u(s+\rho_n,y))\mathbbm{1}_{\{s\leq \rho_{n+1}-\rho_n\}}\dot{W}(dyd(s+\rho_n))\\
            =:&S_n(t,x)+Z_n(t,x).
        \end{split}\label{eq50}
    \end{align}
    By Assumption \ref{assumption3}(A) and Lemma \ref{lem:I_n bound}\eqref{eq16} it follows that for $t\in (0,1)$
    \begin{align}
        |S_n(t)|_{L^\infty}\leq CMt^{-\beta}.\label{eq51}
    \end{align}
    Choose $m>m_0(M)$ so that $T_m=\Big(\frac{CM}{2^{m-2}}\Big)^{\frac{1}{\beta}}<1$ and $S_n(T_m)\leq 2^{m-2}$.
    Theorem \ref{thm:factorization} with
    \begin{align*}
        \varphi(t,x):=\sigma(u(\rho_n+t,x))\mathbbm{1}_{\{t\leq \rho_{n+1}-\rho_n\}},
    \end{align*}
    and the Chebyshev inequality and Theorem \ref{thm:factorization} guarantee that 
    \begin{align}
    \begin{split}
        &\mathbb{P}\Big(\sup_{t\leq T_m}\sup_{x\in D}|Z_n((t+\rho_n),x)|>2^{m-2}\big| |u(\rho_n)|_{L^\infty}=2^m\Big)\\
        \leq & C 2^{-p(m-2)}\mathbb{E}\Big(\sup_{t\leq T_m}\sup_{x\in D}|Z_n((t+\rho_n),x)|^p\big||u(\rho_n)|_{L^\infty}=2^m\Big)\\
        \leq & C 2^{-p(m-2)}T_m^{\frac{(1-\eta)(p-2)}{2}-2\beta}\\
        &\times\mathbb{E}\int_{\rho_n}^{(\rho_n+T_m)\land  \rho_{n+1}}\int_D\int_D|\sigma(u(s,y_1))\sigma(u(s,y_2))|^{\frac{p}{2}}\Lambda(y_1,y_2)dy_1dy_2ds\\
        \leq & C 2^{-p(m-2)}T_m^{\frac{(1-\eta)(p-2)}{2}-2\beta}\\
        &\times\mathbb{E}\sup_{t\in[0,T_m\land(\rho_{n+1}-\rho_n)]}|\sigma(t+\rho_n)|_{L^\infty(D)}^{p-2}\\
        &\times\int_{\rho_n}^{(\rho_n+T_m)\land  \rho_{n+1}}\int_D\int_D|\sigma(u(s,y_1))\sigma(u(s,y_2))|\Lambda(y_1,y_2)dy_1dy_2ds.\\
    \end{split}\label{eq52}
\end{align}
The definition of $\rho_n$ and \eqref{eq9} guarantee that, conditioning on $|u(\rho_n)|_{L^\infty}=2^m$,
\begin{align}
    |\sigma(t+\rho_n)|_{L^\infty(D)}^{p-2}\mathbbm{1}_{\{t\leq \rho_{n+1}-\rho_n\}}\leq C 2^{(1+\frac{1-\eta}{2\beta})(p-2)m}.\label{eq53}
\end{align}
Thus, it follows that
\begin{align}
     \begin{split}
        &\mathbb{P}\Big(\sup_{t\leq T_m}\sup_{x\in D}|Z_n((t+\rho_n),x)|>2^{m-2}\big| |u(\rho_n)|_{L^\infty}=2^m\Big)\\
        \leq & C 2^{(1+\frac{1-\eta}{2\beta})(p-2)m-p(m-2)}T_m^{\frac{(1-\eta)(p-2)}{2}-2\beta}\\
        &\times\mathbb{E}\int_{\rho_n}^{(\rho_n+T_m)\land  \rho_{n+1}}\int_D\int_D|\sigma(u(s,y_1))\sigma(u(s,y_2))|\Lambda(y_1,y_2)dy_1dy_2ds\\
        =& C 2^{(1+\frac{1-\eta}{2\beta})(p-2)m-p(m-2)}\Big(\frac{CM}{2^{m-2}}\Big)^{\frac{(1-\eta)(p-2)}{2\beta}-2}\\
        &\times\mathbb{E}\int_{\rho_n}^{(\rho_n+T_m)\land  \rho_{n+1}}\int_D\int_D|\sigma(u(s,y_1))\sigma(u(s,y_2))|\Lambda(y_1,y_2)dy_1dy_2ds\\
        \leq & C M^{\frac{(1-\eta)(p-2)}{2\beta}-2}\mathbb{E}\int_{\rho_n}^{(\rho_n+T_m)\land  \rho_{n+1}}\int_D\int_D|\sigma(u(s,y_1))\sigma(u(s,y_2))|\Lambda(y_1,y_2)dy_1dy_2ds\\
        \leq &C M^{\frac{(1-\eta)(p-2)}{2\beta}-2}\mathbb{E}\int_{\rho_n}^{  \rho_{n+1}}\int_D\int_D|\sigma(u(s,y_1))\sigma(u(s,y_2))|\Lambda(y_1,y_2)dy_1dy_2ds.
    \end{split}\label{eq54}
\end{align}
The powers of $2$ disappear because 
\begin{align}
    (1+\frac{1-\eta}{2\beta})(p-2)m-p(m-2)-(\frac{(1-\eta)(p-2)}{2\beta}-2)(m-2)=2(p-2),
\end{align}
which does not depend on $m$. We incorporate this constant in $C$. Now we prove that if the event
\begin{align}
    \Big\{\sup_{t\leq T_m}\sup_{x\in D}|Z_n(t,x)|\leq 2^{m-2}\Big\}\label{eq55}
\end{align}
occurs, then $|u(\rho_{n+1})|_{L^\infty}$ falls to $2^{m-1}$ before it can reach $2^{m+1}$. This is because by the previous argument we have
\begin{itemize}
    \item $\sup_{t\leq T_m}|u(\rho_n+t)|_{L^\infty}\leq\sup_{t\leq T_m}|S_n(t,x)|_{L^\infty}+\sup_{t\leq T_m}|Z_n(t,x)|_{L^\infty}\leq  2^m+2^{m-2}<2^{m+1}$.
    \item $|u(\rho_n+T_m)|_{L^\infty}\leq |S_n(T_m,x)|_{L^\infty}+\sup_{t\leq T_m}|Z_n(t,x)|_{L^\infty}\leq  2^{m-2}+2^{m-2}=2^{m-1}$.
\end{itemize}
Thus, $\rho_{n+1}-\rho_n\leq T_m$ a.s. and this implies that if the event \eqref{eq55} occurs then $|u(\rho_n+t)|_{L^\infty}$ falls to the level $2^{m-1}$ before it can rise to the level $2^{m+1}$. Therefore, combine \eqref{eq54} and the argument for \eqref{eq55} then $\eqref{eq49}$ holds.
\end{proof}

\noindent\textit{Proof of Theorem \ref{thm:nonexplosion}}. Fix $M>0$ and let $\rho_n$ be defined from \eqref{eq47} and $m_0$ is chosen from Lemma \ref{lem:double u}. Then we add up the conditional probabilities \eqref{eq49} to see that for any $n\in \mathbb{N}$,
\begin{align}
    \begin{split}
        &\mathbb{P}\Big(|u(\rho_{n+1})|_{L^\infty}=2|u(\rho_n)|_{L^\infty}\text{ and }|u(\rho_n)|_{L^\infty}>2^{m_0}\Big)\\
        =&\sum_{m=m_0}^\infty C_p M^{\frac{(1-\eta)(p-2)}{2\beta}-2}\\
        &\times\mathbb{E}\Big(\int_{\rho_n}^{\rho_{n+1}}\int_D\int_D\sigma(u(s,y_1))\sigma(u(s,y_2))\Lambda(y_1,y_2)dy_1dy_2ds\Big| |u(\rho_{n})|_{L^\infty}=2^{m}\Big)\\
        &\times \mathbb{P}(|u(\rho_n)|_{L^\infty}=2^m)\\
        \leq & \sum_{m=m_0}^\infty C_p M^{\frac{(1-\eta)(p-2)}{2\beta}-2}\mathbb{E}\Big(\int_{\rho_n}^{\rho_{n+1}}\int_D\int_D\sigma(u(s,y_1))\sigma(u(s,y_2))\Lambda(y_1,y_2)dy_1dy_2ds\Big).
    \end{split}\label{eq56}
\end{align}
Then add these probabilities with respect to $n$ it follows that
\begin{align}
    \begin{split}
        &\sum_{n=1}^\infty\mathbb{P}\Big(|u(\rho_{n+1})|_{L^\infty}=2|u(\rho_n)|_{L^\infty}\text{ and }|u(\rho_n)|_{L^\infty}>2^{m_0}\Big)\\
    \leq &  C_p M^{\frac{(1-\eta)(p-2)}{2\beta}-2}\mathbb{E}\Big(\int_{0}^{\tau_M^1}\int_D\int_D\sigma(u(s,y_1))\sigma(u(s,y_2))\Lambda(y_1,y_2)dy_1dy_2ds\Big).
    \end{split}\label{eq57}
\end{align}
By Lemma \ref{lem:quad} \eqref{eq29} it follows that
\begin{align}
    \sum_{n=1}^\infty\mathbb{P}\Big(|u(\rho_{n+1})|_{L^\infty}=2|u(\rho_n)|_{L^\infty}\text{ and }|u(\rho_n)|_{L^\infty}>2^{m_0}\Big)<\infty.\label{eq58}
\end{align}
Therefore, the Borel-Cantelli Lemma guarantees that with probability one, the events $\Big(|u(\rho_{n+1})|_{L^\infty}=2|u(\rho_n)|_{L^\infty}\text{ and }|u(\rho_n)|_{L^\infty}>2^{m_0}\Big)$ only happens a finite number of times. This proves that
\begin{align}
    \mathbb{P}(\tau_M^1<\tau_{\infty}^\infty)=1.\label{eq59}
\end{align}
By Lemma \ref{lem:quad}, let $\gamma\in (0,1)$ and $T>0$ be arbitrary. Choose $M>0$ big enough such that
\begin{align}
    \mathbb{P}(\tau_M^1\geq T\land \tau_\infty^\infty)>1-\gamma.\label{eq60}
\end{align}
Combine \eqref{eq59} and \eqref{eq60} then
\begin{align}
    \mathbb{P}(T<\tau_{\infty}^\infty)>1-\gamma.
\end{align}
Since the choice of $\gamma>0$ is arbitrary, then it follows that
\begin{align}
    \mathbb{P}(T<\tau_{\infty}^\infty)=1.
\end{align}
This is true for arbitrary $T>0$ and therefore,
\begin{align}
    \mathbb{P}(\tau_\infty^\infty=\infty)=1,
\end{align}
and $u(t,x)$ cannot explode in finite time.

\newpage
\bibliographystyle{plain} 
\bibliography{Bioliography}
\end{document}